\documentclass{article}
\usepackage{amsmath, amsthm, amssymb, amsfonts}
\usepackage{mathrsfs}
\usepackage{thmtools}
\usepackage{graphicx}
\usepackage{geometry}
\usepackage{tikz-cd}
\usepackage[british]{babel}
\usepackage{import}
\usepackage{authblk}

\usepackage{hyperref}

\newtheorem{theorem}{Theorem}[section]
\newtheorem{definition}[theorem]{Definition}
\newtheorem{remark}[theorem]{Remark}
\newtheorem{corollary}[theorem]{Corollary}
\newtheorem{proposition}[theorem]{Proposition}

\def\dd{\text{d}}

\def\l{\left}
\def\r{\right}

\def\a{\alpha}
\def\b{\beta}
\def\g{\gamma}

\def\vp{\varphi}

\def\ba{\mathbf{a}}
\def\bb{\mathbf{b}}

\title{Minimising Length of Closed Billiard Trajectories on Hyperbolic Polygons}
\author{John Parker\\Department of Mathematical Sciences,\\ Durham University,\\ South Road, Durham DH1 3LE, UK \and Manvendra Somvanshi\\IISER Mohali,\\Knowledge city, Sector 81, Manauli, \\PO Sahibzada Ajit Singh Nagar, Punjab 140306, India}
\date{\today}

\begin{document}

\maketitle
\abstract{
In a hyperbolic polygon any finite collection of closed billiard trajectories can be assigned an average length function. In this paper, we consider the average length of the collection of cyclically related closed billiard trajectories in even-sided right-angled polygons and the collection of reflectively related closed billiard trajectories in Lambert quadrilaterals with acute angle $\pi/k$. We show that in the former case the average length is minimised by the regular even-sided right-angled polygon, and in the latter case it is minimised by the Lambert quadrilateral with a reflective symmetry about its long axis. We use techniques from Teichm\"{u}ller theory to prove the main theorems. }
\tableofcontents
\section{Introduction}
Let $P\subset \mathbb{D}^2$ be a hyperbolic polygon. We will consider two cases, namely when $P$ has all angles equal and when $P$ is a Lambert quadrilateral. A billiard trajectory on $P$ is the path traced by a point satisfying the following rules: the point moves along geodesics in the polygon and whenever the geodesic path meets an edge of the polygon, it is reflected along a geodesic so that both geodesics make the same angle with that edge. A billiard path is said to be closed if the point returns back to its starting position with the same direction in finite time. In other words, a (closed) billiard trajectory is a (closed) piecewise geodesic $\g$ in $P$ with all vertices of $\g$ lying on the edges of $P$ such that $\g$ makes the same angle on either side with the edge at each vertex.\\

First, consider the case where all the angles of $P$ are equal. If one labels the sides of $P$ with $1,\cdots, k$ in the anticlockwise direction, then every billiard trajectory $\g$ can be assigned a code by noting the number $a_i$ assigned to the edges of $P$ which the trajectory $\g$ hits in order. We call the code $\ba = (a_0,\cdots, a_{n-1})$ associated with $\g$ the \textit{ billiard sequence} of $\g$. Unlike in Euclidean geometry, a billiard sequence uniquely determines the hyperbolic billiard trajectory $\g$. Two billiard trajectories are said to be \textit{cyclically related} in $P$ if there are two anticlockwise labels of $P$ in which both trajectories have the same billiard sequence.\\

Let $\g_{\ba, P} = \{\g_1,\cdots,\g_n\}$ be the collection of all billiard trajectories that are cyclically related to the billiard sequence $\ba$. A collection of closed billiard trajectories is said to \textit{fill} the polygon $P$ if its complement in $P$ is a union of topological discs each of whose whose boundary is made up of either (a) segments of geodesics of the billiard trajectories, (b) geodesic segments of the billiard trajectories and one edge of $P$ including both its vertices, or (c) geodesic segments of the billiard trajectories, two adjacent edges of $P$, and the common vertex of those edges but not containing the other vertices of either edge. In \cite{parker2018minimizing} the following useful result is proved.

\begin{proposition}[\cite{parker2018minimizing}]\label{prop:filling}
    Let $P_0$ be a regular hyperbolic $k$-gon and let $\rho$ denote the anti-clockwise rotation through angle $2\pi/k$ about the center of $P_0$. Let $\g_0$ be a closed billiard trajectory in $P_0$ and let $\g_i = \rho^i(\g_0)$ for $i = 1, \cdots,k-1$. Then $\bigcup_i \g_i$ fills $P_0$. 
\end{proposition}

The average length $L_{\text{avg}}(\ba,P)$ of $\g_{\ba, P}$ is well defined as the average of the hyperbolic lengths of $\{\g_i\}$. The authors of \cite{castle2009billiards} conjecture that: for a given billiard sequence $\ba$ on an ideal $k-$gon the average length function $L_{\text{avg}}(\ba,P)$ is minimised by the regular ideal polygon, $P_0$. This conjecture was proved in \cite{parker2018minimizing}.  

\begin{theorem}[\cite{parker2018minimizing}]\label{thm:ideal}
    In anti-clockwise labeled ideal hyperbolic polygons with $k\geq 3$ sides, the average length function of any family of cyclically related closed billiard trajectories corresponding to a given billiard sequence is uniquely minimised by the regular polygon. 
\end{theorem}

In this paper, we prove a similar result for the case where $P$ is even sided and each internal angle is $\pi/2$. The main result of this paper is the following.

\begin{theorem}\label{thm:main}
    In anti-clockwise labeled, right-angled, even-sided hyperbolic polygons, the average length function of any family of cyclically related closed billiard trajectories corresponding to a given billiard sequence is uniquely minimised by the regular polygon. 
\end{theorem}

The proof of this result closely follows the proof of Theorem~\ref{thm:ideal} given in \cite{parker2018minimizing}. The idea of the proof is as follows. We first observe that by gluing together four isometric copies of the $2k-$gon $P$ together, one can construct a closed Riemann surface of genus $k-1$, which we call the \textit{billiard surface}, $S_P$, corresponding to $P$. Any closed billiard trajectory $\g$ lifts to a collection of closed geodesics on $S_P$. The number of lifts depends entirely on the combinatorial properties of the billiard sequence of $\g$.\\

The space of all billiard surfaces constructed from $k-$gons forms a subspace of the Teichm\"{u}ller space of the closed surface $S_{k-1}$. We endow the Teichm\"{u}ller space with the Weil-Petersson metric. We show that the billiard space is exactly the fixed-point set of an order 2 Weil-Petersson isometry.
This also leads to the conclusion that the billiard space is a geodesically convex subspace of the Teichm\"{u}ller space. Then, using results regarding the convexity of the length function of closed geodesics under the Weil-Petersson metric in \cite{wolpert1987geodesic} and Kerckhoff's result about the length function of filling systems in \cite{kerckhoff1983nielsen} we show that there is a unique point where the global minimum is obtained. This global minimum is then identified as the regular right-angled polygon using the rotational symmetry of the regular polygon.\\
         
Further, we also prove an analogous result for Lambert quadrilaterals. A hyperbolic Lambert quadrilateral is a four-sided polygon with three right angles and an acute angle. Let $Q\subset \mathbb{D}^2$ be an anti-clockwise labeled hyperbolic Lambert quadrilateral with the acute angle being $\pi/k$. From Theorem 2.3.1 in \cite{buser2010geometry} one can conclude that just one side length determines $Q$ completely. There is a unique quadrilateral $Q_0$, which we call the regular Lambert quadrilateral, which is symmetric about its long axis, or equivalently where $a=b$ and $\a=\b$ in Figure~\ref{fig:lambert}.\\
\begin{figure}
    \centering
    \def\svgwidth{0.8\textwidth}
    \begingroup%
  \makeatletter%
  \providecommand\color[2][]{%
    \errmessage{(Inkscape) Color is used for the text in Inkscape, but the package 'color.sty' is not loaded}%
    \renewcommand\color[2][]{}%
  }%
  \providecommand\transparent[1]{%
    \errmessage{(Inkscape) Transparency is used (non-zero) for the text in Inkscape, but the package 'transparent.sty' is not loaded}%
    \renewcommand\transparent[1]{}%
  }%
  \providecommand\rotatebox[2]{#2}%
  \newcommand*\fsize{\dimexpr\f@size pt\relax}%
  \newcommand*\lineheight[1]{\fontsize{\fsize}{#1\fsize}\selectfont}%
  \ifx\svgwidth\undefined%
    \setlength{\unitlength}{265.95878832bp}%
    \ifx\svgscale\undefined%
      \relax%
    \else%
      \setlength{\unitlength}{\unitlength * \real{\svgscale}}%
    \fi%
  \else%
    \setlength{\unitlength}{\svgwidth}%
  \fi%
  \global\let\svgwidth\undefined%
  \global\let\svgscale\undefined%
  \makeatother%
  \begin{picture}(1,0.37919174)%
    \lineheight{1}%
    \setlength\tabcolsep{0pt}%
    \put(0,0){\includegraphics[width=\unitlength,page=1]{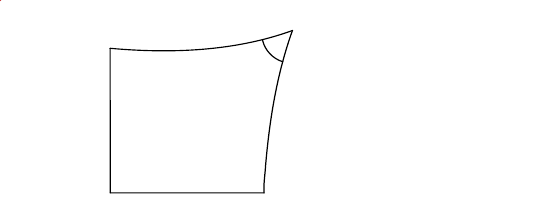}}%
    \put(0.16885548,0.13953031){\color[rgb]{0,0,0}\makebox(0,0)[lt]{\lineheight{1.25}\smash{\begin{tabular}[t]{l}$a$\end{tabular}}}}%
    \put(0.3071583,0.29622047){\color[rgb]{0,0,0}\makebox(0,0)[lt]{\lineheight{1.25}\smash{\begin{tabular}[t]{l}$\b$\end{tabular}}}}%
    \put(0.49271154,0.15911616){\color[rgb]{0,0,0}\makebox(0,0)[lt]{\lineheight{1.25}\smash{\begin{tabular}[t]{l}$\a$\end{tabular}}}}%
    \put(0.30320575,0.00302853){\color[rgb]{0,0,0}\makebox(0,0)[lt]{\lineheight{1.25}\smash{\begin{tabular}[t]{l}$b$\end{tabular}}}}%
    \put(0,0){\includegraphics[width=\unitlength,page=2]{lambert.pdf}}%
  \end{picture}%
\endgroup%

    \caption{On the left is a Lambert quadrilateral. On the right we have a pair of reflectively related billiard trajectories corresponding to the billiard sequence $\ba = (2,3,4)$ and $\bar{\ba} = (1,4,3)$.}
    \label{fig:lambert}
\end{figure}
Let $\g_0$ be a closed billiard trajectory on $Q_0$ and let $\bar{\g}_0$ be the reflection of $\g_0$ along the axis of reflection of $Q_0$. Given any closed billiard trajectory $\g$ on $Q$, using the anti-clockwise labeling we get a billiard sequence $\ba$ corresponding to $\g$. There is a billiard trajectory $\g_0$ on $Q_0$ which corresponds to the billiard sequence $\ba$. Then, let $\bar{\ba}$ be the billiard sequence of the closed billiard trajectory of $\bar{\g}_0$. Let $\bar{\g}$ represent the closed billiard trajectory on $Q$ corresponding to $\bar{\ba}$. We say that $\g$ and $\bar{\g}$ are \textit{reflectively related} and correspond to the billiard sequence $\ba$. Denote pairs of reflectively related closed billiard trajectories corresponding to a billiard sequence $\ba$ by $\g_{\ba, Q}$. We prove the following result about billiard trajectories on Lambert quadrilaterals.

\begin{theorem}\label{thm:lambert}
In anti-clockwise labeled hyperbolic Lambert quadrilateral, the average length function of any pair of reflectively related closed billiard trajectories corresponding to a given billiard sequence is uniquely minimised by the regular Lambert quadrilateral. 
\end{theorem}

The proof of this directly follows from Theorem~\ref{thm:main} as $2k$ copies of $Q$ can be glued together to construct a right-angled $2k-$gon so that the reflectively related billiard trajectories on $Q$ form a family of cyclically related billiard trajectories on the $2k-$gon.\\

In Section~\ref{sec:2} we discuss the construction of billiard surfaces and discuss the lift of closed billiard trajectories on the surface. Then in Section~\ref{sec:3} we discuss basics of Teichm\"{u}ller theory and Weil-Petersson metric, define the Billiard space, and prove some useful properties of the Billiard space. In Section~\ref{sec:4} we introduce the geodesic length function and state some useful results about it. Section~\ref{sec:5} is dedicated to proving Theorem~\ref{thm:main} and Section~\ref{sec:6} is dedicated to proving Theorem~\ref{thm:lambert}. 

\section{Construction of Billiard Surface}\label{sec:2}
Let $P=P_1\subset \mathbb{D}^2$ be a right-angled hyperbolic $2k-$gon. Number the sides of $P_1$ from $1$ to $2k$ in the anticlockwise order and colour the sides of $P$ in blue and red alternatively so that the edge labeled $1$ is blue. Let $P_2$ and $P_3$ be two isometric copies of $P_1$ but with the opposite orientation, and let $P_4$ be an isometric copy of $P_1$ with the same orientation. Gluing $P_1$ and $P_2$, $P_3$ and $P_4$ along their blue edges with the same label gives two isometric surfaces $S$, $S'$ each with $k$ geodesic boundary components. Gluing $S$ and $S'$ along the boundary components with the same label without any twisting results in a closed Riemann surface $S_P$. We call $S_P$ the \textit{billiard surface} corresponding to $P$. From this construction we also get a preferred collection $\Gamma_P$ of non-separating simple closed geodesics on $S_P$ which are colored red or blue and correspond to the sides of the polygons $P_i$, see Figure~\ref{fig:hexagons}. Label the red curves $\a_1,\cdots, \a_k$ and the blue curves $\b_1,\cdots, \b_k$. There are $2k$ points of intersection, $v_1,\cdots, v_{2k}$, of these closed geodesics which correspond to the vertices of the polygons $P_i$. Let $\pi_P:S_P \to P$ be the natural projection of $S_P$ onto $P$.\\
\begin{figure}
    \centering
    \def\svgwidth{0.6\textwidth}
    \begingroup%
  \makeatletter%
  \providecommand\color[2][]{%
    \errmessage{(Inkscape) Color is used for the text in Inkscape, but the package 'color.sty' is not loaded}%
    \renewcommand\color[2][]{}%
  }%
  \providecommand\transparent[1]{%
    \errmessage{(Inkscape) Transparency is used (non-zero) for the text in Inkscape, but the package 'transparent.sty' is not loaded}%
    \renewcommand\transparent[1]{}%
  }%
  \providecommand\rotatebox[2]{#2}%
  \newcommand*\fsize{\dimexpr\f@size pt\relax}%
  \newcommand*\lineheight[1]{\fontsize{\fsize}{#1\fsize}\selectfont}%
  \ifx\svgwidth\undefined%
    \setlength{\unitlength}{718.73248524bp}%
    \ifx\svgscale\undefined%
      \relax%
    \else%
      \setlength{\unitlength}{\unitlength * \real{\svgscale}}%
    \fi%
  \else%
    \setlength{\unitlength}{\svgwidth}%
  \fi%
  \global\let\svgwidth\undefined%
  \global\let\svgscale\undefined%
  \makeatother%
  \begin{picture}(1,0.90727057)%
    \lineheight{1}%
    \setlength\tabcolsep{0pt}%
    \put(0,0){\includegraphics[width=\unitlength,page=1]{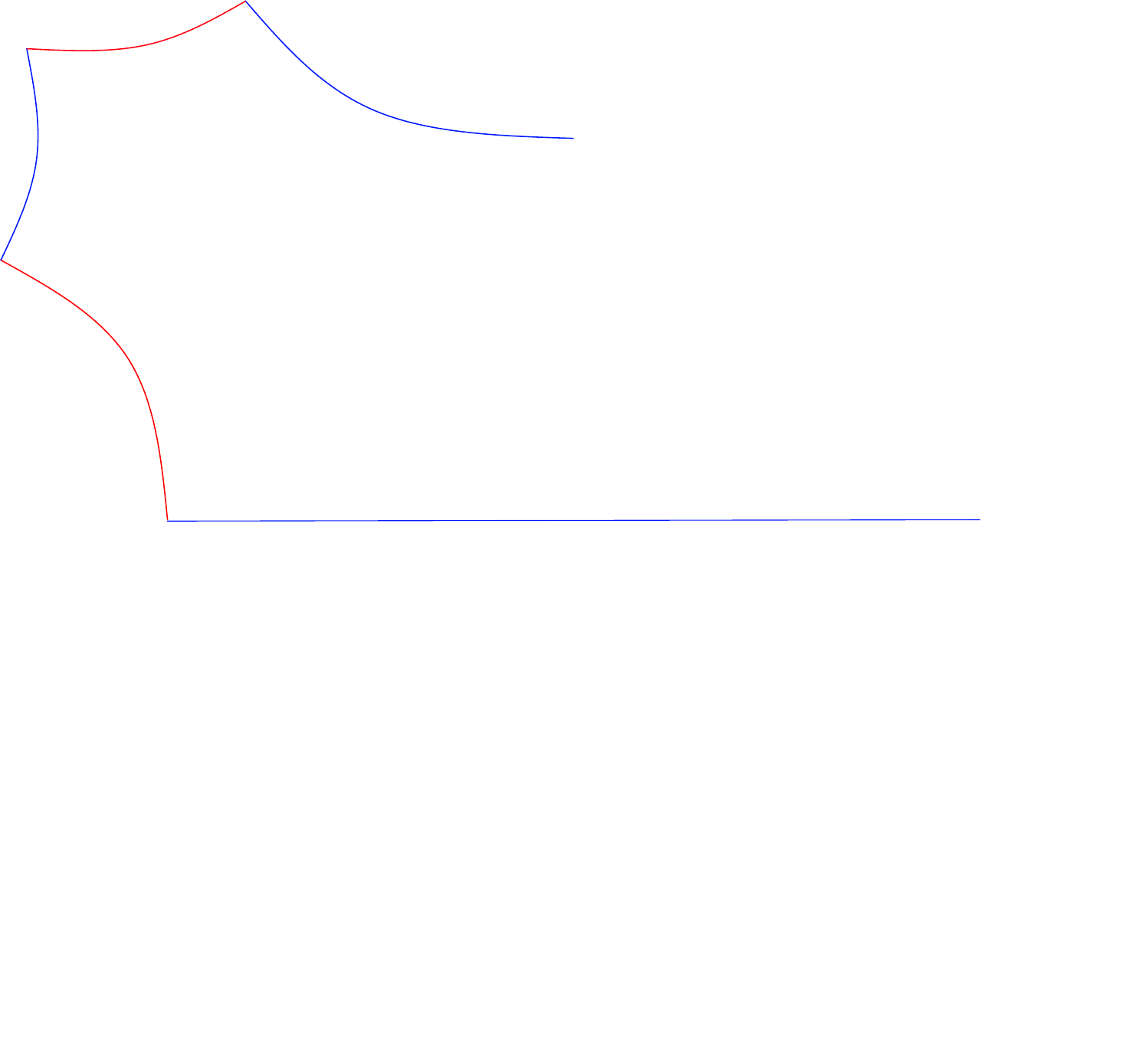}}%
    \put(0.32554437,0.81970563){\color[rgb]{0.02352941,0.02745098,0.04705882}\makebox(0,0)[lt]{\lineheight{1.25}\smash{\begin{tabular}[t]{l}$1$\end{tabular}}}}%
    \put(0.12042888,0.88393496){\color[rgb]{0.02352941,0.02745098,0.04705882}\makebox(0,0)[lt]{\lineheight{1.25}\smash{\begin{tabular}[t]{l}$2$\end{tabular}}}}%
    \put(0.00607567,0.77171468){\color[rgb]{0.02352941,0.02745098,0.04705882}\makebox(0,0)[lt]{\lineheight{1.25}\smash{\begin{tabular}[t]{l}$3$\end{tabular}}}}%
    \put(0.08542928,0.58210209){\color[rgb]{0.02352941,0.02745098,0.04705882}\makebox(0,0)[lt]{\lineheight{1.25}\smash{\begin{tabular}[t]{l}$4$\end{tabular}}}}%
    \put(0.33598283,0.46289485){\color[rgb]{0.02352941,0.02745098,0.04705882}\makebox(0,0)[lt]{\lineheight{1.25}\smash{\begin{tabular}[t]{l}$5$\end{tabular}}}}%
    \put(0.47789599,0.60925185){\color[rgb]{0.02352941,0.02745098,0.04705882}\makebox(0,0)[lt]{\lineheight{1.25}\smash{\begin{tabular}[t]{l}$6$\end{tabular}}}}%
    \put(0.36606911,0.08146776){\color[rgb]{0.02352941,0.02745098,0.04705882}\makebox(0,0)[lt]{\lineheight{1.25}\smash{\begin{tabular}[t]{l}$1$\end{tabular}}}}%
    \put(0.07841588,0.01610203){\color[rgb]{0.02352941,0.02745098,0.04705882}\makebox(0,0)[lt]{\lineheight{1.25}\smash{\begin{tabular}[t]{l}$2$\end{tabular}}}}%
    \put(0.00830051,0.13666976){\color[rgb]{0.02352941,0.02745098,0.04705882}\makebox(0,0)[lt]{\lineheight{1.25}\smash{\begin{tabular}[t]{l}$3$\end{tabular}}}}%
    \put(0.09726037,0.33081648){\color[rgb]{0.02352941,0.02745098,0.04705882}\makebox(0,0)[lt]{\lineheight{1.25}\smash{\begin{tabular}[t]{l}$4$\end{tabular}}}}%
    \put(0.47784188,0.26746172){\color[rgb]{0.02352941,0.02745098,0.04705882}\makebox(0,0)[lt]{\lineheight{1.25}\smash{\begin{tabular}[t]{l}$6$\end{tabular}}}}%
    \put(0.59340813,0.8100901){\color[rgb]{0.02352941,0.02745098,0.04705882}\makebox(0,0)[lt]{\lineheight{1.25}\smash{\begin{tabular}[t]{l}$1$\end{tabular}}}}%
    \put(0.87309381,0.88026515){\color[rgb]{0.02352941,0.02745098,0.04705882}\makebox(0,0)[lt]{\lineheight{1.25}\smash{\begin{tabular}[t]{l}$2$\end{tabular}}}}%
    \put(0.96961637,0.78177841){\color[rgb]{0.02352941,0.02745098,0.04705882}\makebox(0,0)[lt]{\lineheight{1.25}\smash{\begin{tabular}[t]{l}$3$\end{tabular}}}}%
    \put(0.88944231,0.57180297){\color[rgb]{0.02352941,0.02745098,0.04705882}\makebox(0,0)[lt]{\lineheight{1.25}\smash{\begin{tabular}[t]{l}$4$\end{tabular}}}}%
    \put(0.66292846,0.4629272){\color[rgb]{0.02352941,0.02745098,0.04705882}\makebox(0,0)[lt]{\lineheight{1.25}\smash{\begin{tabular}[t]{l}$5$\end{tabular}}}}%
    \put(0.59124182,0.08221242){\color[rgb]{0.02352941,0.02745098,0.04705882}\makebox(0,0)[lt]{\lineheight{1.25}\smash{\begin{tabular}[t]{l}$1$\end{tabular}}}}%
    \put(0.87444373,0.01166166){\color[rgb]{0.02352941,0.02745098,0.04705882}\makebox(0,0)[lt]{\lineheight{1.25}\smash{\begin{tabular}[t]{l}$2$\end{tabular}}}}%
    \put(0.97860092,0.12949377){\color[rgb]{0.02352941,0.02745098,0.04705882}\makebox(0,0)[lt]{\lineheight{1.25}\smash{\begin{tabular}[t]{l}$3$\end{tabular}}}}%
    \put(0.89020016,0.3315051){\color[rgb]{0.02352941,0.02745098,0.04705882}\makebox(0,0)[lt]{\lineheight{1.25}\smash{\begin{tabular}[t]{l}$4$\end{tabular}}}}%
    \put(0.23335799,0.67146776){\color[rgb]{0.02352941,0.02745098,0.04705882}\makebox(0,0)[lt]{\lineheight{1.25}\smash{\begin{tabular}[t]{l}$P_1$\end{tabular}}}}%
    \put(0.68294613,0.67987904){\color[rgb]{0.02352941,0.02745098,0.04705882}\makebox(0,0)[lt]{\lineheight{1.25}\smash{\begin{tabular}[t]{l}$P_3$\end{tabular}}}}%
    \put(0.25085474,0.26690322){\color[rgb]{0.02352941,0.02745098,0.04705882}\makebox(0,0)[lt]{\lineheight{1.25}\smash{\begin{tabular}[t]{l}$P_2$\end{tabular}}}}%
    \put(0.71397421,0.27371994){\color[rgb]{0.02352941,0.02745098,0.04705882}\makebox(0,0)[lt]{\lineheight{1.25}\smash{\begin{tabular}[t]{l}$P_4$\end{tabular}}}}%
    \put(0,0){\includegraphics[width=\unitlength,page=2]{hexagon-gluing.pdf}}%
  \end{picture}%
\endgroup%
    \caption{Gluing four copies of a right-angled hexagon to obtain the billiard surface. The blue edges in $P_1$ (resp. $P_3$) are glued to the blue edges in $P_2$ (resp. $P_4$), and the red edges in $P_1$ (resp. $P_2$) are glued to the red edges in $P_3$ (resp. $P_4$)}.
    \label{fig:hexagons}
\end{figure}
The billiard surface $S_P$ has 3 natural maps coming from permuting the polygons $P_i$. Let $J_P:S_P\to S_P$ be the map induced from applying the permutation $(12)(34)$ to the polygons $P_1,\cdots, P_4$; and $K_P:S_P\to S_P$ be the map induced from applying the permutation $(13)(24)$. Both of these are order 2 anti-holomorphic maps of $S_P$ and they preserve the red and blue geodesics. The composition $J_PK_P = K_PJ_P: S_P\to S_P$ is an order 2 holomorphic map, $F_P$, which fixes the points $v_1,\cdots, v_k$. From this description it follows that $F_P$ is the hyperelliptic involution on $S_P$.\\

Let $\g_1$ be a billiard trajectory and its associated billiard sequence be $\ba = (a_0,\cdots, a_{n-1})$. Then $\g_{\ba, P} = \{\g_1, \cdots, \g_n\}$ is the collection of all cyclically related billiard paths with billiard sequence $\ba$. Finally, in this section we count the number of closed geodesics in $\pi_P^{-1}(\g_i)$.

\begin{proposition}\label{prop:counting}
  Let $\g_{\ba, P}$ be as described above. The number of closed geodesic curves in the lift $\pi_P^{-1}(\g_i)$ of the billiard trajectory $\g_i\in \g_{\ba, P}$ is $1,2,$ or $4$.
\end{proposition}

\begin{proof}
  Let the billiard sequence corresponding to $\g_i$ be $(a_1,\cdots, a_n)$. We use the following convention to construct closed curves in the lift of $\g_i$: We start in the polygon $P_1$ and trace out the billiard trajectory using the billiard sequence. Each $a_j$ corresponds to a ``transition" from some $P_i$ to some $P_j$. After going through all $(a_1,\cdots, a_n)$ suppose that the trajectory ends in some $P_j$. In the case $P_j=P_1$ the curve closes up to form a closed geodesic $\tilde{\g}_i$. Using the isometries we get at most four lifts in total: $\tilde{\g}_i, J_P(\tilde{\g}_i), K_P(\tilde{\g}_i), J_PK_P(\tilde{\g}_i)$. Note here that $3$ lifts is not possible since, if $\g_i=I(\g_i)$ for $I\in \{J_P,K_P,J_PK_P\}$ then it is easy to show that $I', I''\in\{J_P,K_P,J_PK_P\}-\{I\}$ are such that $I'(\g_i) = I''(\g_i)$. Thus the number of lifts here is $1,2$, or $4$.

  Now suppose that $J_i\neq J_1$. One of the order 2 maps $J_P,K_P, J_PK_P$ interchanges $P_1$ and $P_j$, call this isometry $I$. Using the isometry $I$ twice, one can see that by tracing out $\mathbf{a}$ twice on $S_P$ we get a closed curve $\tilde{\g}_i$. Again, using the isometries $J_P,K_P,J_PK_P$ we obtain other lifts of $\g_i$, but since $I$ is one of them and fixes $\g_i$, in this case there are at most $2$ distinct lifts.
\end{proof}

\begin{remark}
  In the genus $2$ case $J_PK_P$ fixes every simple closed geodesic and $J_P(\a) = K_P(\a)$ for any simple closed geodesic $\a$. Thus, it is easy to construct examples where there are only 1 or 2 lifts in the genus $2$ case.
 For example, the billiard trajectory corresponding to the sequence $(1,4)$ on a right-angled hexagon, $P$, has a unique lift to the genus-$2$ surface $S_P$. Also, the upper bound of $4$ is achieved in every surface of genus $g\geq 3$.
\end{remark}
\section{Teichm\"{u}ller Space and the Billiard Space}\label{sec:3}
Let $P_0\subset \mathbb{D}^2$ be the unique regular right-angled $2k-$gon and $R_0$ be the corresponding billiard surface. By construction, there is a curve system $\Gamma_0$ on $R_0$, as described in the previous section. By definition, a regular polygon has a rotational symmetry, i.e. the rotation, $\rho$, by angle $\pi/k$ is holomorphic transformation of $P_0$. This induces a holomorphic automorphism of the surface $R_0$, which we also denote by $\rho$. Let $J_0:R_0\to R_0$ and $F_0:R_0\to R_0$ be the order 2 isometries induced from the permutation of the polygons as in the previous section.

\begin{definition}
    A marking of a Riemann surface $S$ by the surface $R_0$ is an orientation preserving homeomorphism $f:R_0\to S$. Two markings $(S,f)$ and $(S',f')$ are said to be equivalent if $f'f^{-1}:S\to S'$ is isotopic to a holomorphic map. The equivalence classes, $[S,f]$, of markings is called a marked surface. The Teichm\"{u}ller space $T(R_0)$ is the collection of marked surfaces.     
\end{definition}

The Teichm\"{u}ller space $T(R_0)$ has a natural complex structure on it. Given any orientation preserving homeomorphism $\vp: R_0\to R_0$ it induces a map $\vp_*:T(R_0)\to T(R_0)$ which maps $[S,f] \mapsto [S, f\circ \vp]$. If $S$ is a Riemann surface then let $S^*$ be the same surface as $S$ but with opposite orientation and let $\iota:S\to S^*$ be the identity map which reverses the orientation of $S$.

\begin{definition}
    Define the following two maps $\mathcal{J}, \mathcal{F}:T(R_0)\to T(R_0)$ as
    \begin{align}
        \mathcal{J}([S,f:R_0\to S]) = (S^*, \iota f J_0: R_0 \to S^*),\\
        \mathcal{F}([S,f:R_0\to S]) = (S, f F_0: R_0 \to S)
    \end{align}
\end{definition}

The even sided right-angled $2k-$gon $P$, where $k\geq 4$, can always be decomposed into right-angled hexagons using the following unique orthogonal geodesics between the blue edges: $(1,5)$, $(1,7)$, $\cdots$, $(1,2k-3)$. This gives $k-3$ geodesic arcs in $P$ which we color green. Following the construction of $S_P$, these green geodesic arcs form $2(k-3)$ simple closed geodesics on $S_P$ which we label $\delta_1,\cdots, \delta_{k-3}$ and $\delta'_1,\cdots, \delta'_{k-3}$ so that $F_P(\delta_i) = \delta'_i$ for each $i=1,\cdots, k-3$, see Figure~\ref{fig:genus-3}. The curve system, $\mathcal{P}_P$, consisting of the red and green geodesics gives a pants decomposition of $S_P$ with the blue curves being the seam of the pants.\\ 
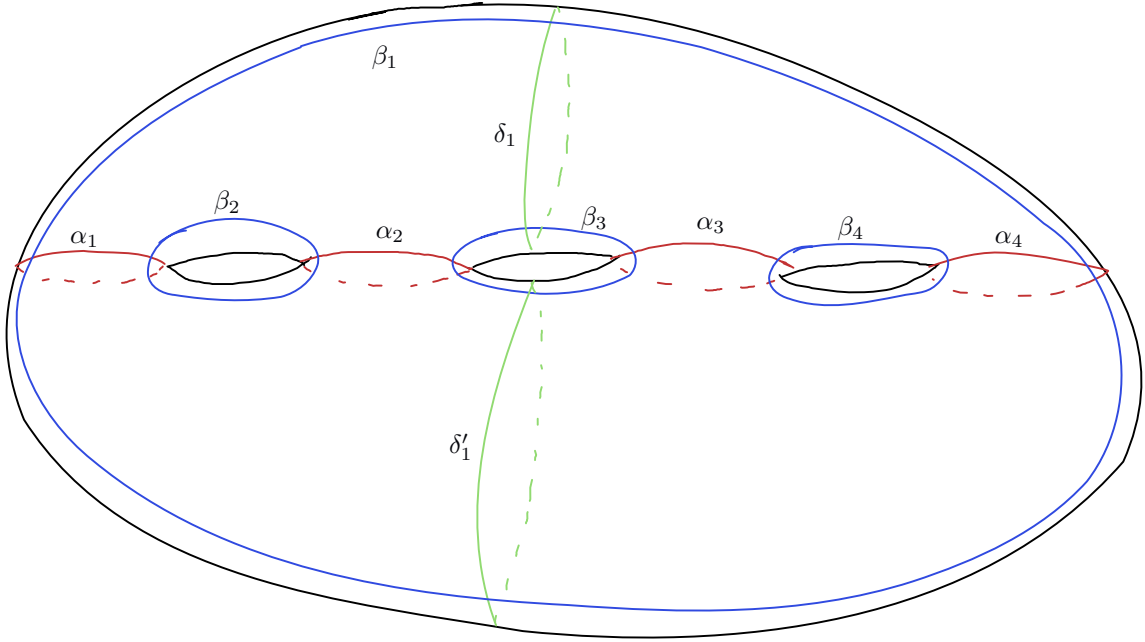
\begin{figure}
    \centering
    \def\svgwidth{\textwidth}
    \begingroup%
  \makeatletter%
  \providecommand\color[2][]{%
    \errmessage{(Inkscape) Color is used for the text in Inkscape, but the package 'color.sty' is not loaded}%
    \renewcommand\color[2][]{}%
  }%
  \providecommand\transparent[1]{%
    \errmessage{(Inkscape) Transparency is used (non-zero) for the text in Inkscape, but the package 'transparent.sty' is not loaded}%
    \renewcommand\transparent[1]{}%
  }%
  \providecommand\rotatebox[2]{#2}%
  \newcommand*\fsize{\dimexpr\f@size pt\relax}%
  \newcommand*\lineheight[1]{\fontsize{\fsize}{#1\fsize}\selectfont}%
  \ifx\svgwidth\undefined%
    \setlength{\unitlength}{833.94208885bp}%
    \ifx\svgscale\undefined%
      \relax%
    \else%
      \setlength{\unitlength}{\unitlength * \real{\svgscale}}%
    \fi%
  \else%
    \setlength{\unitlength}{\svgwidth}%
  \fi%
  \global\let\svgwidth\undefined%
  \global\let\svgscale\undefined%
  \makeatother%
  \begin{picture}(1,0.56038462)%
    \lineheight{1}%
    \setlength\tabcolsep{0pt}%
    \put(0,0){\includegraphics[width=\unitlength,page=1]{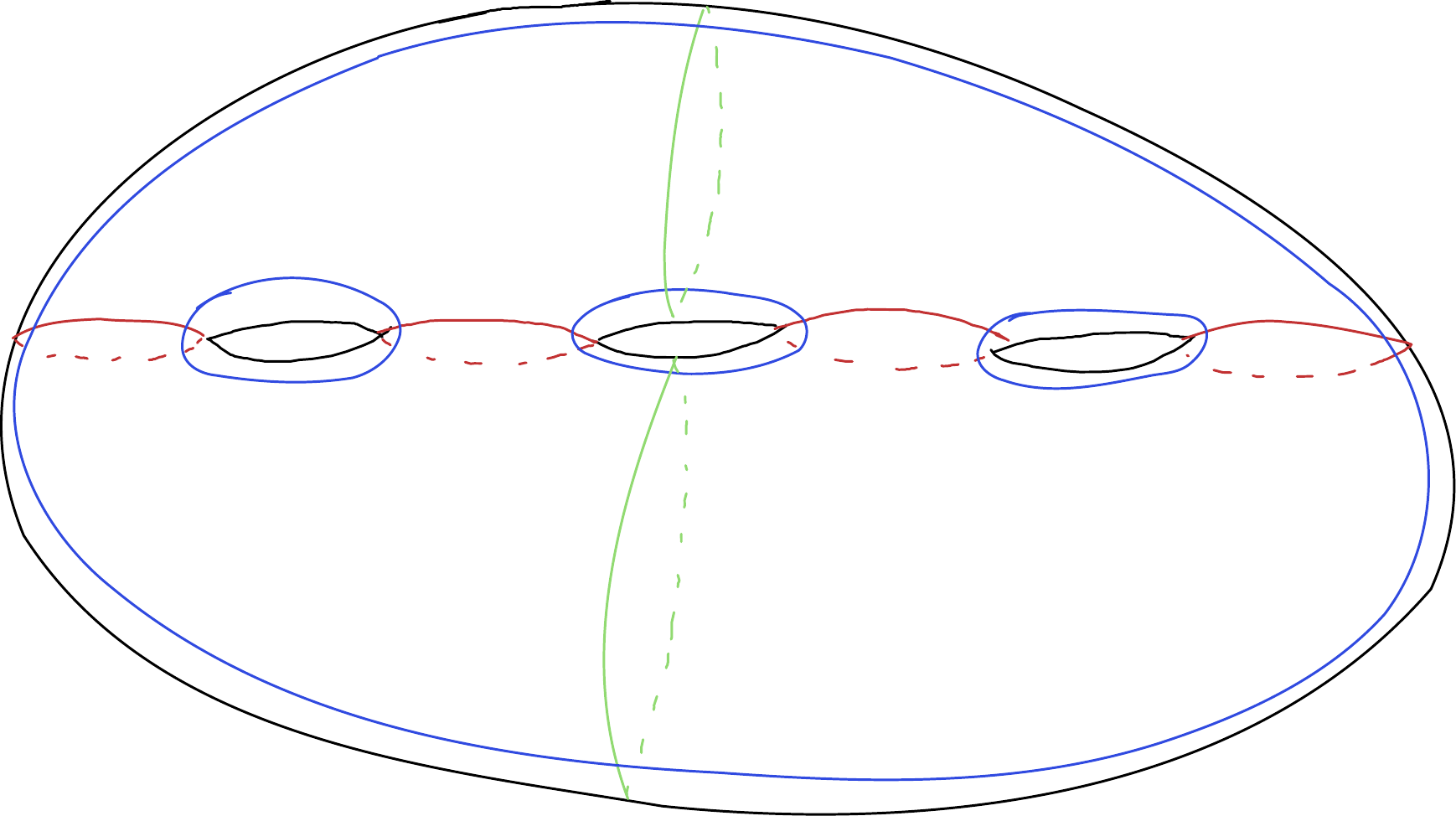}}%
    \put(0.0568176,0.3482828){\color[rgb]{0.02352941,0.02745098,0.04705882}\makebox(0,0)[lt]{\lineheight{1.25}\smash{\begin{tabular}[t]{l}$\a_1$\end{tabular}}}}%
    \put(0.18312768,0.37823464){\color[rgb]{0.02352941,0.02745098,0.04705882}\makebox(0,0)[lt]{\lineheight{1.25}\smash{\begin{tabular}[t]{l}$\b_2$\end{tabular}}}}%
    \put(0.32556278,0.3522267){\color[rgb]{0.02352941,0.02745098,0.04705882}\makebox(0,0)[lt]{\lineheight{1.25}\smash{\begin{tabular}[t]{l}$\a_2$\end{tabular}}}}%
    \put(0.50600575,0.36424405){\color[rgb]{0.02352941,0.02745098,0.04705882}\makebox(0,0)[lt]{\lineheight{1.25}\smash{\begin{tabular}[t]{l}$\b_3$\end{tabular}}}}%
    \put(0.60782311,0.360712){\color[rgb]{0.02352941,0.02745098,0.04705882}\makebox(0,0)[lt]{\lineheight{1.25}\smash{\begin{tabular}[t]{l}$\a_3$\end{tabular}}}}%
    \put(0.73173147,0.35525457){\color[rgb]{0.02352941,0.02745098,0.04705882}\makebox(0,0)[lt]{\lineheight{1.25}\smash{\begin{tabular}[t]{l}$\b_4$\end{tabular}}}}%
    \put(0.87030268,0.3480825){\color[rgb]{0.02352941,0.02745098,0.04705882}\makebox(0,0)[lt]{\lineheight{1.25}\smash{\begin{tabular}[t]{l}$\a_4$\end{tabular}}}}%
    \put(0.32204577,0.50557286){\color[rgb]{0.02352941,0.02745098,0.04705882}\makebox(0,0)[lt]{\lineheight{1.25}\smash{\begin{tabular}[t]{l}$\b_1$\end{tabular}}}}%
    \put(0.42933949,0.43734515){\color[rgb]{0.02352941,0.02745098,0.04705882}\makebox(0,0)[lt]{\lineheight{1.25}\smash{\begin{tabular}[t]{l}$\delta_1$\end{tabular}}}}%
    \put(0.39022449,0.16305628){\color[rgb]{0.02352941,0.02745098,0.04705882}\makebox(0,0)[lt]{\lineheight{1.25}\smash{\begin{tabular}[t]{l}$\delta'_1$\end{tabular}}}}%
  \end{picture}%
\endgroup%
    \caption{The curve systems $\Gamma_P$ and $\mathcal{P}_P$ in the case where $P$ is an octagon.}
    \label{fig:genus-3}
\end{figure}

Let $\mathcal{P}_0$ be the pants decomposition on $R_0$ as constructed above. This induces a pants decomposition $f(\mathcal{P}_0)$ on the marked surface $[S,f]\in T(R_0)$. The lengths, $\ell_{[S.f]}(\a)$, of the geodesic representative of $\a$ and the twist parameter $\tau_{[S,f]}(\a)$, measured with respect to the geodesic representatives of the seams $\{f(\b_i)\}_{i=1}^{k}$, for all $\a\in f(\mathcal{P}_0)$ gives a coordinate system on $T(R_0)$, called the \textit{Fenchel-Nielsen coordinates} with respect to $\mathcal{P}_0$.\\

The Teicm\"{u}ller space, $T(R_0)$, can be endowed with a complex structure with a symplectic form, $\omega_{WP}$ called the Weil-Petersson form. In the Fenchel-Nielsen coordinates with respect to $\mathcal{P}_0$, denoted $(\ell_1,\cdots, \ell_{3k-6}, \tau_1,\cdots, \tau_{3k-6})$, the Weil-Petersson symplectic form can be expressed using Wolpert's formula \cite{wolpert1985weil} as:
\begin{align}
    \omega_{WP} = \sum_{j=1}^{3k-6} \dd \tau_j \wedge \dd \ell_j 
\end{align}
The Weil-Petersson form induces a Riemannian metric, $g_{WP}$, on the Teichm\"{u}ller space. It is well known that the action of the extended mapping class group, $\text{Map}^\pm(R_0)$ which is defined as $\text{Homeo}(R_0)/\text{Homeo}_0(R_0)$, leaves the metric $g_{WP}$ invariant, see \cite{imayoshi1992teichmuller}. The Riemannian metric $g_{WP}$ is also uniquely geodesic.\\     

Finally, in this section we define the billiard space of $R_0$ and prove some useful properties.

\begin{definition}
    The billiard space, $B(R_0)$, is a subset of the Teichm\"{u}ller space, $T(R_0)$, where $x = [S,f]\in B(R_0)$ if
    \begin{enumerate}
        \item $\tau_x(\a) = 0$ for all $\a\in f(\mathcal{P}_0)$.
        \item The length of the geodesic representative of the image of $\delta_i$ and $\delta'_i$ in $S$ is the same, \textit{i.e.} $\ell_x(f(\delta_i)) = \ell_x(f(\delta'_i))$ for all $i=1,\cdots, k-3$.
    \end{enumerate}
\end{definition}

\begin{remark}\label{rem:marked-billiard-surface}
    There are quasi-conformal homeomorphisms between the right-angled hexagons in $R_0 - (\Gamma_0\cup \mathcal{P}_0)$ and the right-angled hexagons in $S_P - (\Gamma_P \cup \mathcal{P}_P)$, so that the edges and vertices map to the edges and vertices of the corresponding hexagons. These maps can be chosen such that upon gluing back they induce a quasi-conformal map $f_P:R_0\to S_P$ which maps the graph $\Gamma_0\cup \mathcal{P}_0$ to the graph $\Gamma_P \cup \mathcal{P}_P$, preserving edges and vertices. This results in a marked surface $[S_P,f_P]\in T(R_0)$. It is clear from this construction that $[S_P,f_P]\in B(R_0)$.  
\end{remark}

\begin{proposition}\label{prop:fixed-point-set}
    The subspace $B(R_0)$ is exactly the simultaneous fixed points of $\mathcal{J}$ and $\mathcal{F}$.
\end{proposition}

\begin{proof}
  Let $[S,f]\in B(R_0)$. Then $f( \mathcal{P}_0)$ is a pants decomposition of $[S,f]$ and $fF_0( \mathcal{P}_0)$ is a pants decomposition of $[S,fF_0]$. The Fenchel-Nielsen coordinates with respect to $\mathcal{P}_0$ of $[S,f]$ is so that all the twists parameters are zero and $\ell(f(\delta_i)) = \ell(f(\delta'_i))$. Note that $F_0(\a) =\a$ for any curve $\a\in \Gamma_0$ but with orientation reversed and $F_0(\delta_i) = \delta'_i$. Thus, all the twist parameters of $[S,fF_0]$ are zero. Since all lengths and twists are preserved we conclude that $[S,f] = [S,fF_0]$. \\

  Similarly to the proof above, it can be shown that $B(R_0)\subset \text{Fix}(\mathcal{J})$ as $J_0$ fixes every curve in the pants decomposition 
  $ \mathcal{P}_0$ but with reversed orientation.\\

  Suppose that $x = [S,f]\in \text{Fix}( \mathcal{F})\cap \text{Fix}( \mathcal{J})$. As $ \mathcal{F}$ fixes $x$ it follows that $\tau_{x}(\a_i) = 0$, $\tau_x(\delta_i) = -\tau_x(\delta'_i)$, and $\ell(f(\delta_i)) = \ell(f(\delta'_i))$, since $F_0$ reverses the orientation of all $\a_i$ and maps $\delta_i$ to $\delta'_i$. As $ \mathcal{J}$ fixes $x$ it follows that $\tau_x(\delta_i) = \tau_x(\delta'_i) = 0$ since $J_0$ preserves all $\delta_i$. Thus, we have that $x\in B(R_0)$. This completes this proof.
\end{proof}

\begin{proposition}
Let $x= [S,f]$ be a marked surface. Then $x\in B(R_0)$ if and only if $x = [S_P,f_P]$.
\end{proposition}

\begin{proof}
    The converse is true from the discussion in Remark~\ref{rem:marked-billiard-surface}. Suppose that $x\in B(R_0)$. Then from 
    Proposition~\ref{prop:fixed-point-set} it follows that $x$ is fixed by $\mathcal{J}, \mathcal{F}$ which means that $fF_0f^{-1}$ and 
    $\iota fJ_0f^{-1}$ are homotopic to conformal maps $g, h$ respectively. The maps $g$, $h$ and $gh$ preserve each curve in $f(\Gamma_0)$ while 
    permuting the four even sided polygon faces in $S-f(\Gamma_0)$, implying that they are all isometric to an even sided hyperbolic polygon $P$. 
    Since the edges of these faces glue together in pairs to form $S$ and the four polygons meet at each vertex in $f(\Gamma_0)$, it follows that 
    all the angles of the polygon $P$ are right angles. Thus $S$ is isometric to the surface $S_P$ via some isometry $I$ and the curve system 
    $\Gamma_P = If(\Gamma_0)$. This means that both $If$ and $f_P$ have the same action on the graph $\Gamma_0$. Since $\Gamma_0$ satisfies the 
    assumptions of Alexander's method it follows that $If$ is homotopic to $f_P$. Thus $[S,f] = [S_P, f_P]$.   
\end{proof}

\begin{proposition}
    $B(R_0)$ is a convex subspace of $T(R_0)$. Moreover, $B(R_0)$ is an isotropic submanifold of $(T(R_0),\omega_{WP})$.
\end{proposition}

\begin{proof}
    The maps $ \mathcal{F}$ and $\mathcal{J}$ are an isometry of $g_{WP}$ as they are induced from the action of the (extended) mapping classes 
    $[F_0]$ and $[J_0]$ respectively. From Proposition~\ref{prop:fixed-point-set} it follows that $ B(R_0)$ is exactly the intersection of the fixed 
    point set of $ \mathcal{F}$ and $ \mathcal{J}$, both of which are isometries of $g_{WP}$. Any geodesic connecting two points in $B(R_0)$ can be 
    mapped to some other geodesic connecting the same two points by both $\mathcal{J}$ and $\mathcal{F}$ but since the metric is uniquely geodesic 
    it follows that any such geodesic lies in the fixed point set of both isometries $\mathcal{J}, \mathcal{F}$. Therefore, it lies in $B(R_0)$. 
    This proves $B(R_0)$ is convex.\\

    Since $\mathcal{J}_*(\omega_{WP}) = - \omega_{WP}$, see \cite{imayoshi1992teichmuller}, and $B(R_0)\subset \text{Fix}(\mathcal{J})$ it follows that $\omega_{WP} = 0$ when restricted to the Billiard space $B(R_0)$. Hence $B(R_0)$ is an isotropic submanifold of $T(R_0)$. 
\end{proof}
\section{The geodesic length function}\label{sec:4}

\begin{definition}
    Let $\tilde{\g} = \{\tilde{\g}_1,\cdots, \tilde{\g}_N\}$ be a collection of closed geodesics on $R_0$ and let $x=[S,f]\in T(R_0)$ be some marked surface. The associated geodesic length function $L_{\tilde{\g}}:T(R_0)\to (0,\infty)$ is defined as
    $$
    L_{\tilde{\g}}(x) = \sum_{i=1}^N \ell_x(f(\tilde{\g}_i))
    $$
\end{definition}

The following theorem of Wolpert in \cite{wolpert1987geodesic} states the continuity and convexity properties of the geodesic length function and will be used in proof of Theorem~\ref{thm:main}.

\begin{theorem}\label{thm:Wolpert}
    Let $\tilde{\g}\subset R_0$ be a finite collection of closed geodesics and let $L_{\tilde{\g}}$ be the associated geodesic length function. Then $L_{\tilde{\g}}$ is continuous and strictly convex along every Weil-Petersson geodesic.   
\end{theorem}

Let $\ba = (a_0, \cdots, a_{n-1})$ be a billiard sequence on the polygon $P$ and let $\g_{\ba, P}$ be the collection of all cyclically related billiard trajectories with billiard sequence $\ba$. If $\g, \g'\in \g_{\ba, P_0}$ are two billiard trajectories which are cyclically related then the number of closed geodesics in the lift of $\g'$ is same as that of $\g$ since $\g' = \rho^\ell(\g)$ for some power $\ell$ and $\pi_{P_0} \rho = \rho \pi_{P_0}$. Then using the map $f_P:R_0\to S_P$, it is straightforward to see that any two billiard trajectories in $\g_{\ba, P}$ have the same number of closed geodesics in their lifts. Define $\tilde{\g}_{\ba, P}$ as the collection of closed geodesics in $\pi_P^{-1}(\g)$ for all $\g\in \g_{\ba, P}$. 
From Proposition~\ref{prop:counting} and the above discussion it is clear that $\tilde{\g}_{\ba, P}$ contains either $n$, $2n$ or $4n$ closed geodesics counted with possible repetition.\\

In the case when the lift of $\g\in \g_{\ba, P}$ has one closed geodesic in its lift, the length of that geodesic is $4\ell(\g)$. In the case when $\g$ has two closed geodesics in its lift it follows that the length of each of those closed geodesics is $2\ell(\g)$. Finally, in the last case, when there are $4$ closed geodesics in the lift of $\g$, the length of each geodesic is $\ell(\g)$. From this it follows that the average length function associated to billiard sequence $\ba = (a_0,\cdots, a_{n-1})$,
\begin{align}
L_{\text{avg}}(\ba, P) = \frac{1}{|\g_{\ba, P}|}\sum_{\g\in \g_{\ba, P}} \ell(\g)    
\end{align}
can be expressed, in any of the three cases, as
\begin{align}\label{eq:length}
    L_{\text{avg}}(\ba, P) = \frac{1}{4n} L_{\tilde{\g}_{\ba, P_0}}([S_P, f_P])
\end{align}
where we have used that $\g_{\ba, P} = f_P(\tilde{\g}_{\ba, P_0})$.

\begin{proposition}\label{prop:invariance}
Let $\ba$ be a billiard sequence and $L_{\tilde{\g}_{\ba, P_0}}$ be the geodesic length function associated to $\g_{\ba, P_0}$. Then for any $x = [S, f:R_0\to S]\in T(R_0)$
\begin{align}
    L_{\tilde{\g}_{\ba, P_0}}(x) = L_{\tilde{\g}_{\ba, P_0}}(\rho(x)) = L_{\tilde{\g}_{\ba, P_0}}(\mathcal{J}(x)) = L_{\tilde{\g}_{\ba, P_0}}(\mathcal{F}(x))
\end{align}
\end{proposition}

\begin{proof}
    Note that for any billiard sequence $\ba$ the map $I\in \{\rho, J_0,F_0\}$ permutes the closed geodesics in $\tilde{\g}_{\ba, P_0}$. Let $\mathcal{I}$ denote corresponding map induced by $I$. Then,
    \begin{align}
        L_{\tilde{\g}_{\ba, P_0}}(\mathcal{I}(x)) = \sum_{\tilde{\g}\in \tilde{\g}_{\ba, P_0}} \ell(fI(\tilde{\g})) = \sum_{\tilde{\g}\in \tilde{\g}_{\ba, P_0}} \ell(f(\tilde{\g})) = L_{\tilde{\g}_{\ba, P_0}}(x)
    \end{align}
    Here, the second equality follows from the fact that $I$ permutes the closed geodesics. This completes the proof.
\end{proof}
\section{Proof of Theorem \ref{thm:main}}\label{sec:5}
Let $\ba$ be a billiard sequence, then Theorem~\ref{thm:main} claims that
\begin{align}
    L_{\text{avg}}(\ba, P_0) < L_{\text{avg}}(\ba, P_0)
\end{align}
From \ref{eq:length} it follows that this is equivalent to the claim
\begin{align}
    L_{\tilde{\g}_{\ba, P_0}}(x_0) < L_{\tilde{\g}_{\ba, P}}(x)
\end{align}
In order to prove this, we use the following result from \cite{kerckhoff1983nielsen}.

\begin{theorem}\label{thm:kerckhoff}
    Let $\tilde{\g}$ be a finite collection of closed geodesics on $R_0$ and $L_{\tilde{\g}}$ be the associated geodesic length function. If $\tilde{\g}$ is filling, that is $R_0 - \tilde{\g}$ is a disjoint collection of discs and once punctured discs, then $L_{\g}$ is a proper function.
\end{theorem}

Combining Theorems~\ref{thm:Wolpert} and~\ref{thm:kerckhoff}, the following result can be proved:

\begin{corollary}
    Let $\tilde{\g}$ be a finite collection of closed geodesics which fills $R_0$. Then the associated geodesic length function $L_{\tilde{\g}}:T(R_0)\to (0,\infty)$ achieves its global minimum at a unique point $x_{\text{min}} \in T(R_0)$.   
\end{corollary}

Coming back to the problem at hand, if we can show that the curve system $\tilde{\g}_{\ba, P_0}$ fills $R_0$ then the geodesic length function, $L_{\tilde{\g}_{\ba, P_0}}$, assumes a global minimum at a unique point $x_{\text{min}}\in T(R_0)$. Using Proposition~\ref{prop:invariance}, it follows that 
$$
L_{\tilde{\g}_{\ba, P_0}}(x_{\text{min}}) = L_{\tilde{\g}_{\ba, P_0}}(\mathcal{J}(x_{\text{min}})) = L_{\tilde{\g}_{\ba, P_0}}(\mathcal{F}(x_{\text{min}})) = L_{\tilde{\g}_{\ba, P_0}}(\rho_*(x_{\text{min}}))
$$
Thus, from the uniqueness of the global minimum, it follows that $x_{\text{min}}$ is fixed by $\mathcal{J}, \mathcal{F},$ and $\rho_*$. The following two propositions complete the proof of Theorem~\ref{thm:main}. 

\begin{proposition}
    The collection $\tilde{\g}_{\ba, P_0}$ of closed geodesics on $R_0$ is filling.
\end{proposition}

\begin{proof}
    Since the surface $R_0$ is constructed by gluing together four isometric copies of $P_0$ we can write:
    \begin{align}
        R_0- \bigcup_{\tilde{\g}\in \tilde{\g}_{\ba, P_0}} \tilde{\g} =& \l(P_0 - \bigcup_{\g\in \g_{\ba, P_0}} \g \r) \cup J_0\l(P_0 - \bigcup_{\g\in \g_{\ba, P_0}} \g \r)\nonumber \\ &\cup F_0\l(P_0 - \bigcup_{\g\in \g_{\ba, P_0}} \g \r) \cup F_0J_0\l(P_0 - \bigcup_{\g\in \g_{\ba, P_0}} \g \r)
    \end{align}
    From Proposition \ref{prop:filling} it follows that $\g_{\ba, P_0}$ fills in $P_0$. Thus, $\tilde{\g}_{\ba, P_0}$ cuts $R_0$ into discs or once punctured discs.
\end{proof}

\begin{proposition}\label{prop:fixed-pt}
The only simultaneous fixed point of the maps $\mathcal{J}, \mathcal{F}$, and $\rho_*$ is $x_0 = [R_0, 1_{R_0}]$.
\end{proposition}

\begin{proof}
    Suppose that $x\in T(R_0)$ is a fixed point of $\mathcal{F}$, $\mathcal{J}$ and $\rho_*$. It follows from Proposition~\ref{prop:fixed-point-set} that $x\in B(R_0)$, since $x$ is fixed by $\mathcal{F}$ and 
    $\mathcal{J}$. This means that there is a right-angled even sided polygon $P$ such that $x = [S_P, f_P]$ and there is a special curve system $\Gamma_P$ such that $S_P-\Gamma_P$ contains four disjoint copies of $P$.\\

  Being a fixed point of $\rho_*$ means that $g_0=f_P\rho f_P^{-1}$ is homotopic to some isometry $g_1:P\to P$. By construction, $f_P$ maps edges and vertices in $\Gamma_0$ to the corresponding edges and vertices in $\Gamma_P$ then $g_1$ maps the curves $\a_i$ to $\b_{i+1}$ for $i\leq k$, $\a_{k+1}$ to $\b_1$, and $\b_i$ to $\a_i$. Thus, $g_1(P_1) = P_i$ for some $i$, but since $P_2$, $P_3$ have opposite orientations to $P_1$, we have that either $g_1(P_1) = P_1$ or $g_1(P_1)=P_4$. In the latter case, composing $g_1$ with the hyperelliptic involution $F_P$ maps $P_1$ back to $P_1$. Thus, without loss of generality we assume that $g_1$ preserves $P_1$.\\

  Looking only at $P=P_1$ inside the disc now, we have an isometry $g_1:P\to P$ which maps each vertex to the next and each side to the next. By the Brouwer fixed point theorem, there is a fixed point $z_0$ of $g_1$ in the interior of the polygon. Consider the unique geodesics connecting $z_0$ and two vertices $v_i$ and $v_{i+1}$. Then $g_1$ isometrically maps this triangle to the triangle with vertices $z_0,v_{i+1},v_{i+2}$. Thus the adjacent side lengths are equal. Hence $P$ is regular.
\end{proof}
\section{Billiards on hyperbolic Lambert quadrilaterals}\label{sec:6}

Let $Q\subset \mathbb{D}^2$ be a Lambert quadrilateral with the internal acute angle $\pi/k$. Let $Q_1,\cdots, Q_{2k}$ be $2k$ copies of $Q$. They can be glued together to get a right-angled $2k-$gon, $P_Q$. Note that this is symmetric under rotation by $2\pi/k$. Further, gluing together four copies of $P_Q$ gives us a surface, $S_{P_Q}$, of genus $k-1$ as in Section~\ref{sec:2}. Note that the billiard surface corresponding to $Q_0$ is exactly the surface $R_0$, as gluing together $Q_0$ gives the regular right-angled $2k-$gon, $P_0$. As before, there is a quasi-conformal map $f_Q:R_0\to S_P$ which maps the sides and vertices of the quadrilaterals $Q_0$ to the sides and vertices of $Q$. Moreover, the billiard space of Lambert quadrilaterals with angle $\pi/k$ is exactly the subspace of $B(R_0)$ which is fixed by $\rho_*^2:T(R_0)\to T(R_0)$.\\

There is a natural projection $\pi_Q:P_Q\to Q$. Let $\g_{\ba, Q_0}$ be a pair of reflectively related closed billiard trajectories in $Q_0$ having the billiard sequence $\ba$. Then $\pi_{Q_0}^{-1}(\g_{\ba, Q_0})$ is a collection of closed billiard trajectories $\{\g_1,\cdots, \g_n\}$ in $P_0$. Due to the rotational symmetry of $P_0$ it follows that the family $\{\g_1,\cdots, \g_n\}$ are all cyclically related, and moreover, are the \textit{all} the cyclically related billiard trajectories having some billiard sequence $\bb$. It follows that for any $Q$ a pair of reflectively related closed billiard trajectories, $\g_{\ba, Q} = \{\g, \bar{\g}\}$, lift to the family of closed billiard trajectories $\g_{\bb, P_Q}$ having some billiard sequence $\bb$. If $\g_{\ba, Q} = \{\g, \bar{\g}\}$, we define the average length function as
\begin{align}
    L_{\text{avg}}(\ba, Q) = \frac{1}{2}\l(\ell(\g) + \ell(\bar{\g})\r)
\end{align}
The total length of the closed billiard trajectories in $\g_{\bb, P_Q} = \{\g_1,\cdots, \g_n\}$ will just be $2k(\ell(\g) + \ell(\bar{\g}))$. Thus it follows that $n L_{\text{avg}}(\bb, P_Q) = 4k L_{\text{avg}}(\ba, Q)$. Since, from theorem \ref{thm:main}, it follows that $L_{\text{avg}}(\bb, P)$ is uniquely minimized at $P_0 = P_{Q_0}$, and thus $L_{\text{avg}}(\ba, Q)$ is uniquely minimized at $Q_0$. This completes the proof of Theorem~\ref{thm:lambert}.
\bibliographystyle{alpha}
\bibliography{Billiards}
\end{document}